\documentclass[12pt,oneside]{amsart}
\usepackage{graphicx} 
\usepackage{amsmath}            
\usepackage{amsfonts}              
\usepackage{amsthm}                
\usepackage{verbatim}              
\usepackage{indentfirst}           

\usepackage[english]{babel}

\DeclareTextSymbolDefault{\ae}{T1}

\usepackage{enumitem}
\usepackage{mathtools} 
\usepackage{amssymb}
\usepackage[all]{xy}
\usepackage{caption}
\usepackage{subcaption}
\usepackage{setspace}
\usepackage{mathrsfs}
\usepackage{pinlabel}
\usepackage[percent]{overpic}
\usepackage{float}
\usepackage{longtable}
\usepackage[colorlinks=false, hidelinks]{hyperref}
\usepackage[none]{hyphenat}

\numberwithin{equation}{section}

\newtheorem{thm}{Theorem}[section]

\newtheorem{definition}[thm]{Definition}

\makeatletter
\newcommand\addmark@section{\S}
\renewcommand*\@seccntformat[1]{%
  \csname addmark@#1\endcsname
  \csname the#1\endcsname\quad
}
\makeatother

\newtheorem{lem}[thm]{Lemma}
\newtheorem*{definition*}{Definition}
\newtheorem{cor}[thm]{Corollary}

\theoremstyle{remark}
\newtheorem{rmk}{Remark}

\theoremstyle{remark}

\theoremstyle{definition}

\newcommand{\congto}{\xrightarrow{\raisebox{-0.5ex}[0ex][0ex]{$\sim$}}}

\makeatletter
\@namedef{subjclassname@2020}{\textup{2020} Mathematics Subject Classification}
\makeatother

\begin{document}

\title[Dynamical Arakelov-Green's functions]{Arakelov--Green's functions for dynamical systems on projective varieties}
\author{Nicole R.~Looper}
\subjclass[2020]{Primary: 11G10, 37P30. Secondary: 14G40, 14K15, 37P55}
\date{}
\maketitle

\begin{abstract} We introduce functions associated to polarized dynamical systems that generalize averages of the dynamical Arakelov-Green's functions for rational functions due to Baker and Rumely \cite{BakerRumely:potentialthy}. For a polarized dynamical system $X\to X$ over a product formula field, we prove an Elkies-style lower bound for these functions evaluated on the adelic points of $X$. As an application, we prove a Lehmer-type lower bound on the canonical height of a non-torsion point $P$ on an abelian variety $A/K$, where $K$ is a product formula field having perfect residue fields at its completions (for instance, $K$ may be a number field or the function field of a curve over $\mathbb{C}$ or $\mathbb{F}_p$). For $A$ of dimension $g$, the lower bound has the form \[\hat{h}(P)\ge\frac{C}{D^{2g+3}(\log D)^{2g}},\] where $C=C(A,K,\hat{h})>0$, $D=[K(P):K]\ge 2$, and $P\in A(\overline{K})\setminus A(\overline{K})_{\textup{tors}}$ is not contained in a torsion translate of an abelian subvariety of $A$ having everywhere potential good reduction.
    
\end{abstract}

\section{Introduction}

In the arithmetic dynamics of rational functions $f:\mathbb{P}_K^1\to\mathbb{P}_K^1$ over a number field $K$, the dynamical Arakelov-Green's function introduced by Baker--Rumely \cite{BakerRumely:potentialthy} plays a significant role in results concerning points of small canonical height. From providing an avenue to proving the equidistribution of small sequences \cite[Theorem 10.38]{BakerRumely:book} to governing the prime support and fields of definition of preperiodic points \cite{Benedetto:preperiodic,HindrySilverman,Looper:UBCunicrit,Looper:UBCpolys,Petsche:Ih}, these bivariate functions on $\mathbb{P}^1\times\mathbb{P}^1\setminus\Delta$ find ample use. Their utility stems from the confluence of two fortuitous properties: first, their link to potential theory in one variable, and second, the simple fact that the logarithmic distance term that they involve is subject to the product formula. The latter fact forms what is by far the most frequent pivot of arguments which apply these Green's functions to global problems. 

The main purpose of this paper is to introduce generalized Arakelov-Green's functions associated to polarized dynamical systems in arbitrary dimension, and to prove a lower bound for these functions (Theorem \ref{thm:basisbound}) in terms of the geometry of homogeneous filled Julia sets. Strictly speaking, these generalized Arakelov-Green's functions generalize the \emph{averages} of Arakelov-Green's functions that are so often used in the dimension one setting, rather than the functions themselves. Theorem \ref{thm:basisbound} has the following adelic corollary when $K$ is a product formula field. (See Definition \ref{def:prodformula} for the definition of a product formula field. Examples include number fields and the function field of a normal irreducible projective variety.) 

\begin{thm}[cf.~Theorem \ref{thm:mainthm}]\label{mainthm:AGhigherdim} Let $K$ be a product formula field with normalized system of absolute values $|\cdot|_v$ satisfying the product formula, and let $X$ be a projective variety over $K$ with $\mathrm{dim}(X)\ge1$. Suppose $\phi:X\to X$ is a polarized dynamical system on $X$, defined over $K$ (cf.~\S2), and let $f:\mathbb{P}_K^N\to\mathbb{P}_K^N$ be a morphism of $\mathbb{P}_K^N$ that extends $\phi$ via an embedding $\iota:X\hookrightarrow\mathbb{P}_K^N$. For each $n\in\mathbb{Z}_+$, let $\mathcal{B}_n$ be any basis for $H^0(X,\iota^*(\mathcal{O}(n)))$, and let $d_{\mathcal{B}_n}(\mathcal{K}_v\cap\pi_v^{-1}(X_v))$ be as in Definition \ref{def:dBn}. Then there is a constant $C=C(f)$ such that for all $n\ge 2$, \[\sum_{v\in M_K}\log d_{\mathcal{B}_n}(\mathcal{K}_v\cap\pi_v^{-1}(X_v))\le\frac{C\log n}{n}.\] \end{thm}

In \S\ref{section:Lehmer}, we use Theorem \ref{mainthm:AGhigherdim} along with pluripotential-theoretic results due to Boucksom--Eriksson \cite[Theorem D]{BoucksomEriksson} and Boucksom--Gubler--Martin \cite[Corollary D]{BoucksomGublerMartin} (see also \cite[Theorem B]{GGJKM}) in order to prove a Lehmer-type theorem for the N\'{e}ron--Tate height of non-torsion points on an abelian variety $A/K$, where $K$ is a product formula field having perfect residue fields at its completions.

\begin{thm}[cf.~Theorem \ref{thm:Lehmerbdabvar}]\label{introthm:Lehmerbdabvar} Let $K$ be a product formula field having perfect residue fields at its completions (for example, $K$ may be a number field or the function field of a curve over $\mathbb{C}$ or $\mathbb{F}_p$), and let $A/K$ be an abelian variety of dimension $g$. Fix a symmetric ample line bundle $\mathcal{L}$ on $A$ and let $\hat{h}_\mathcal{L}$ be the associated canonical height function. There is a constant $C=C(A,K,\mathcal{L})>0$ such that for any $P\in A(\overline{K})\setminus A(\overline{K})_{\textup{tors}}$ that is not contained in any torsion translate of an abelian subvariety of $A\times_K K'$ having everywhere potential good reduction for some finite extension $K'/K$, we have \[\hat{h}_\mathcal{L}(P)\ge\frac{C}{D^{2g+3}(\log D)^{2g}},\] where $D=[K(P):K]\ge2$.\end{thm}

In particular, for any geometrically simple abelian variety having at least one nonarchimedean place of stable bad reduction, Theorem \ref{introthm:Lehmerbdabvar} gives a height lower bound of the form $C/(D^{2g+3}(\log D)^{2g})$, where $P$ is any non-torsion point. The best known lower bound in the number field case is due to Masser \cite{Masser1, Masser3}, and is of the form $C/(D^{2g+1}(\log D)^{2g})$ for general $A$ (cf.~\cite{GaudronRemond,HindrySilverman:Lehmer,LooperSilverman,Masser2} for the development as well as related results). Theorem \ref{introthm:Lehmerbdabvar} also generalizes and strengthens a lower bound on the Bernoulli part of the height \cite{LooperSilverman} in the case of abelian surfaces over function fields of characteristic $0$. Outside of \cite{LooperSilverman}, there does not appear to be prior progress in the function field setting.

In the study of higher-dimensional dynamics, Arakelov-Green's functions have been used to prove arithmetic results surrounding points of small canonical height on higher genus curves embedded in their Jacobians (see for example \cite{Cinkir,DeJong,DeJong:NT,LooperSilvermanWilms}). The centerpiece of the approach here is Zhang's admissible pairing \cite{Zhang}, which is essentially the data of a Green's function at each place of a product formula field. This setup of a curve embedded into its Jacobian may be viewed as a hybrid of the dimension one and the higher-dimensional settings, since despite the fact that the ambient dynamical system is of dimension greater than one, it is the geometry of curves and their associated Green's functions that is being leveraged. Outside of this setting, progress in higher-dimensional arithmetic dynamics has largely been made via Arakelov theory. For example, Yuan's well-known equidistribution theorem \cite[Theorem 3.1]{Yuan} for polarized dynamical systems (which has since been generalized to \cite[Theorem 5.4.3]{YuanZhang}) utilizes Zhang's fundamental inequality on arithmetic volumes of metrized line bundles in conjunction with a variational argument originally due to Szpiro--Ullmo--Zhang \cite[p.~343--344]{SUZ}. 
While the framework of Arakelov theory is powerful and flexible, quantitative and/or uniform results in the higher-dimensional setting have nonetheless been generally hard to obtain. By constrast, arithmetic dynamics in dimension one is at this point well-populated with a number of such results (e.g., \cite{DeMarcoKriegerYe1,DeMarcoKriegerYe2,DeMarcoMavraki,HindrySilverman,Ingram,Looper:UBCpolys,Petsche}). A common reason behind the relative success in this latter setting is the fact that Elkies-type lower bounds on averages of Arakelov-Green's functions are known to hold, and moreover may often be made uniform and explicit in families of dynamical systems. For the dynamical Arakelov-Green's functions introduced by Baker and Rumely, such a lower bound is due to Baker \cite{Baker}. For the Arakelov-Green's functions underlying Zhang's admissible pairing, the analogous lower bound is a result due to Baker--Rumely \cite{BakerRumely:harmonic}, coming from harmonic analysis on metrized graphs. Each of these lower bounds is of the form \begin{equation}\label{keylowerbound}\sum_{P_i\ne P_j\in T}g(P_i,P_j)\ge-Cn\log n\end{equation} for any set $T=\{P_1,\dots,P_n\}$ of $n\ge2$ pairwise distinct points, and a constant $C$ that depends in a ``nice'' way on the space or dynamical system in question. Theorem \ref{thm:basisbound} provides an analogue to (\ref{keylowerbound}) in arbitrary dimension (an analogue which could be viewed as a generalization if different bases of $H^0(\mathbb{P}^1,\mathcal{O}(n))$ had been used in defining $g$ in (\ref{keylowerbound})). As will be explained in \S2, the effect of choosing different bases cancels out when one sums over all places of a product formula field, yielding Theorem \ref{mainthm:AGhigherdim}.

The proof strategy of Theorem \ref{mainthm:AGhigherdim} draws its inspiration from Baker's proof \cite[Theorem 1.1]{Baker} for rational functions on $\mathbb{P}^1$. Given a morphism $f:\mathbb{P}^N\to\mathbb{P}^N$ of degree at least $2$ over a valued field $K$ and a homogeneous lift $F:\mathbb{A}^{N+1}\to\mathbb{A}^{N+1}$ of $f$, one constructs, for each integer $n\ge 1$, a special basis $\mathcal{F}$ of $H^0(\mathbb{P}^N,\mathcal{O}(n))$ consisting of various products of coordinate functions of iterates of $F$. One then shows that these polynomials take relatively small values when evaluated on the homogeneous filled Julia set of $F$. Baker's proof that his distinguished collection of polynomials forms a basis may be viewed through three different lenses: one from intersection theory, one from syzygy theory, and one from elimination theory. When applied to the higher-dimensional setting, the intersection-theoretic approach proves the linear independence of a small subset $\mathcal{F}'$ of the analogous collection $\mathcal{F}$, a subset whose order grows as $o(h^0(\mathbb{P}^N,\mathcal{O}(n)))$ as $n$ tends to infinity. While $\mathcal{F}'$ may seem adequate from a number of different points of view, there are basic desirable properties of the transfinite diameter (which is the concept underlying the Arakelov-Green's functions $g_n$) that are lost when one works with $\mathcal{F}'$, such as translation-invariance. On the other hand, from a more algebraic angle, the reader familiar with the theory of syzygies will see a potential approach to generalizing (\ref{keylowerbound}) via Koszul complexes. Alas, despite appearances, it is far from clear how to rule out certain degenerate syzygies that would prevent the analogous $\mathcal{F}$ from forming a basis of $H^0(\mathbb{P}^N,\mathcal{O}(n))$. Finally, Macaulay's resultant and its concomitant ``Main Theorem of Elimination Theory" \cite[\S3]{Macaulay} points to a natural collection $\mathcal{F}''$ of elements of $H^0(\mathbb{P}^N,\mathcal{O}(n))$ closely related to $\mathcal{F}$, one which may be shown to span $H^0(\mathbb{P}^N,\mathcal{O}(n))$. (This perspective also elucidates the difficulties that arise with the approach using syzygies, as the relevant Macaulay matrix in the dimension one setting is square, while in general it has more rows than columns. Its left null space is the syzygy space; see \cite[\S2-3]{BatselierDreesenDeMoor} for more details.) Discarding any redundant vectors leads to a basis $H(n)$ for $H^0(\mathbb{P}^N,\mathcal{O}(n))$, and accordingly we prove a lower bound of the form (\ref{keylowerbound}) for the particular basis $H(n)$ (Theorem \ref{thm:basisbound}). However, as there is no way to know which vectors must be discarded---in particular, different vectors may need to be discarded for different maps---we are unable for the moment to relate the determinants corresponding to $H(n)$ to the generalized Vandermonde determinants corresponding to the standard monomial basis. From a number-theoretic perspective, this obstacle is of little consequence: arithmetic applications in fact tend to rely on the aforementioned adelic version of (\ref{keylowerbound}), and this adelic version is independent of the choice of basis for $H^0(\mathbb{P}^N,\mathcal{O}(n))$. For further details, we refer the reader to \S\ref{section:mainthm}.
\newline

\textbf{Acknowledgements}: The author would like to thank Matt Baker, Joe Silverman, and Kevin Tucker for helpful conversations about this project, as well as Rob Benedetto and Laura DeMarco for useful comments on a draft of this article. This research was supported by NSF grant DMS-2302586 and a Sloan Research Fellowship.

\section{Definitions, Background, and Basis Formation}

Let $K$ be a field. A \emph{polarized dynamical system} (over $K$) is a triple $(X,\phi,\mathcal{L})$ consisting of a projective variety $X/K$, an endomorphism $\phi:X\to X$, and an ample line bundle $\mathcal{L}$ on $X$ such that $\phi^*(\mathcal{L})\cong\mathcal{L}^d$ for some integer $d\ge2$. By a well-known theorem of Fakhruddin \cite[Corollary 2.2]{Fakhruddin}, such a polarized dynamical system $\phi$ may be extended to a degree $d$ morphism $f$ of a suitable ambient projective space $\mathbb{P}^N$. In other words, there is an integer $N$, an embedding $\iota:X\hookrightarrow\mathbb{P}^N$ and a morphism $f:\mathbb{P}^N\to\mathbb{P}^N$ defined over $K$ such that the diagram \begin{equation}\label{eqn:Fakhruddin}
\xymatrix{
	{X} \ar[rr]^{\phi} \ar@{^{(}->}[d]^{\iota}
	&& {X} \ar@{^{(}->}[d]^{\iota} \ar@{->}[d]\\
	{\mathbb{P}^N} \ar[rr]^f
	&& {\mathbb{P}^N} 
}
\end{equation} is commutative. The embedding $\iota:X\to\mathbb{P}(H^0(X,\mathcal{L}^m))$ is given by evaluating a fixed basis for the global sections of $\mathcal{L}^m$ at the points of $X$ for an appropriate $m\in\mathbb{Z}_+$. By the Weil height machine \cite[Theorem B.3.2]{HindrySilverman:book}, for any height $h_{X,\mathcal{L}^m}$ with respect to $\mathcal{L}^m$ on $X$, we have \[h(\iota(P))=h_{X,\mathcal{L}^m}(P)+O(1)\] for all $P\in X(\overline{K})$, where $h$ is the standard Weil height on $\mathbb{P}^N$. Hence by Tate's telescoping argument \cite[pp.~195--196]{HindrySilverman:book}, \[\hat{h}_f(\iota(P)):=\lim_{n\to\infty}\frac{1}{d^n}h(f^n(\iota(P)))=\lim_{n\to\infty}\frac{1}{d^n}h_{X,\mathcal{L}^m}(\phi^n(P))=:\hat{h}_{X,\mathcal{L}^m,\phi}(P).\]Since by \cite[Theorem B.4.1(ii)]{HindrySilverman:book} we have $\hat{h}_{X,\mathcal{L}^m,\phi}(P)=m\hat{h}_{X,\mathcal{L},\phi}(P)$ for all $P\in X(\overline{K})$, this yields \begin{equation}\label{eqn:htmult}\hat{h}_f(\iota(P))=m\hat{h}_{X,\mathcal{L},\phi}(P).\end{equation} The polarized dynamical system $(X,\mathcal{L},\phi)$ satisfying (\ref{eqn:Fakhruddin}) may be viewed as a pair $(X,f)$ consisting of a morphism $f:\mathbb{P}^N\to\mathbb{P}^N$ and a subvariety $X\subseteq\mathbb{P}^N$ such that $f|_X$ has image contained in $X$, i.e., such that $f$ leaves $X$ invariant. In the reverse direction, any such pair $(X,f)$ induces the polarized dynamical system $(X,\iota^*(\mathcal{O}(1)),f|_X)$, where $\iota:X\hookrightarrow\mathbb{P}^N$ is the inclusion map. We thus frame our polarized dynamical systems as consisting of these pairs $(X,f)$. We assume throughout this article that $\textup{dim}(X)>0$.

\subsection{Generalized Arakelov-Green's functions and the transfinite diameter} 

Let $(X,f)$ be a polarized dynamical system of degree $d\ge2$ over $K$, satisfying a diagram as in (\ref{eqn:Fakhruddin}). A \emph{homogeneous lift} of $f$ is a map $F:\mathbb{A}^{N+1}\to\mathbb{A}^{N+1}$ such that if $F=(F_0,F_1,\dots,F_N)$, then $F_i$ is a degree $d$ homogeneous polynomial in $K[x_0,\dots,x_N]$ for all $0\le i\le N$ and $f=[F_0:F_1:\cdots:F_N]$. If $K$ is a valued field with absolute value $|\cdot|$ and $P\in\mathbb{A}^{N+1}(K)$, we write \[||F(P)||=\max\{|F_0(P)|,\dots,|F_N(P)|\}.\] We define the \emph{homogeneous filled Julia set} of $F$ to be the set \[\mathcal{K}=\{P\in\mathbb{A}^{N+1}(K): ||F^n(P)||\not\to\infty \text{ as }n\to\infty\}.\] Throughout, we will let \[\pi:\mathbb{A}^{N+1}\setminus\{(
0,\dots,0)\}\to\mathbb{P}^N\] denote the natural projection. When $\mathbb{A}^{N+1}$ and $\mathbb{P}^N$ are defined over a completion $K_v$ of a valued field $K$, we will write this map as $\pi_v$. 

We now build Arakelov--Green's functions associated to this dynamical system. Let $c(n)=\textup{dim}\,H^0(X,\iota^*(\mathcal{O}(n)))$. Given a choice of bases $\{\mathcal{B}_n\}_{n=1}^\infty$ of $H^0(X,\iota^*(\mathcal{O}(n)))$ respectively, there is a naturally induced such function \[g_n:X^{c(n)}\to\mathbb{R}\cup\{\infty\}\] for each $n$, constructed as follows. For a homogeneous lift $F:\mathbb{A}^{N+1}\to\mathbb{A}^{N+1}$ of $f$, let $\widehat{H}_F:\mathbb{A}^{N+1}\setminus\{0,\dots,0\}$ be given by \[\widehat{H}_F(P)=\lim_{n\to\infty}\frac{1}{d^n}\log||F^n(P)||.\] (That this limit exists is shown in \cite[proof of Theorem 9]{KawaguchiSilverman:Green}.) Let $\textup{Res}(F)$ be the Macaulay resultant of the coordinate functions $F_0,\dots,F_N$ defining $F$, and write \[r(F)=\frac{1}{d(d-1)(N+1)}\log|\mathrm{Res}(F)|.\] For $P\in X(K)$, write $\widetilde{P}$ for a lift of $P$ to $\mathbb{A}^{N+1}(K)$, with coordinates in $K$. The Arakelov-Green's function $g_n:X^{c(n)}\to\mathbb{R}\cup\{\infty\}$ associated to the basis $\mathcal{B}_n$ of $H^0(X,\iota^*(\mathcal{O}(n)))$ is given by \begin{equation}\label{eqn:AGfn}g_n\left(P_1,\dots,P_{c(n)}\right)=\frac{1}{c(n)}\sum_{i=1}^{c(n)}\widehat{H}_F(\widetilde{P_i})-\frac{1}{n\cdot c(n)}\log\left|\mathrm{det}\left(\eta_j(\widetilde{P_i})\right)_{\eta_j\in\mathcal{B}_n}\right|+r(F).\end{equation} It is readily verified that for a fixed $\mathcal{B}_n$, this function is independent of the choice of lifts $\widetilde{P_i}$ and $F$. Moreover, if $\mathcal{B}_n'$ is another basis for $H^0(X,\iota^*(\mathcal{O}(n)))$, then its associated Arakelov-Green's function differs by a constant from that associated to $\mathcal{B}_n$. In the special case of $X=\mathbb{P}^1$, $\iota=\textup{id}$, and $\mathcal{B}_n$ the standard monomial basis of $H^0(\mathbb{P}^1,\mathcal{O}(n))$, this $g_n$ coincides with the function $\mathscr{A}(g):(\mathbb{P}^1)^{n+1}\to\mathbb{R}\cup\{\infty\}$ given by \[\mathscr{A}(g)(P_1,\dots,P_{n+1})=\frac{1}{n(n+1)}\sum_{i\ne j}g(P_i,P_j)\] for $g$ the dynamical Arakelov-Green's functions introduced by Baker and Rumely (see \cite[\S3.4]{BakerRumely:potentialthy} and \cite[\S10.2]{BakerRumely:book}). 

We will show that there is a special sequence of bases $\{\mathcal{B}_n\}$ with associated Arakelov-Green's functions $g_n$ such that there is a constant $C=C(X,f)$ satisfying \begin{equation}\label{eqn:lowerboundshape} g_n\left(P_1,\dots,P_{c(n)}\right)\ge-\frac{C\log n}{n}\end{equation} for all $n\ge2$. By itself, this statement is contentless: since we always have the trivial lower bound \[g_n(P_1,\dots,P_{c(n)})\ge-\textup{diam}(\mathcal{K}), \] one can easily rescale any given bases $\mathcal{B}_n$ in order to artificially achieve (\ref{eqn:lowerboundshape}). However, it turns out that when $K$ is a product formula field (cf.~Definition \ref{def:prodformula}), and $g_{n,v}$ is the Arakelov-Green's function associated to our special $\mathcal{B}_n$ at the place $v$, then there is a $C$ such that for \emph{all} $v\in M_K$, and all $n\ge 2$, \begin{equation}\label{eqn:lowerboundforv}g_{n,v}\left(P_1,\dots,P_{c(n)}\right)\ge -\frac{C\log n}{n}.\end{equation} (This is the content of Theorem \ref{thm:basisbound}.) Such a lower bound may not simply be achieved by rescaling any given basis, as the product formula implies that the local effects of any rescaling cancel out across all places. This provides some justification for viewing these lower bounds as living more naturally in a global, rather than a local setting. One may take $C=0$ in (\ref{eqn:lowerboundforv}) for all but finitely many places $v$ (cf.~Remark \ref{rmk:KS}), so the corresponding global lower bound becomes \begin{equation}\label{eqn:lowerboundglobal}\sum_{v\in M_K}\max_{\left(P_1,\dots,P_{c(n)}\right)\in X(\mathbb{C}_v)^{c(n)}}g_{n,v}\left(P_1,\dots,P_{c(n)}\right)\ge -\frac{C'\log n}{n}\end{equation} for $C'=C'(X,f)$.

Finally, we reformulate the problem of proving (\ref{eqn:lowerboundforv}) and (\ref{eqn:lowerboundglobal}) in terms of quantitative convergence to the transfinite diameter of the $v$-adic homogeneous filled Julia set $\mathcal{K}_v$ of a lift $F:\mathbb{A}_K^{N+1}\to\mathbb{A}_K^{N+1}$ of $f$. As $g_{n,v}\left(P_1,\dots,P_{c(n)}\right)$ is well-defined independently of the choices of lifts $\widetilde{P_i}\in\mathbb{A}_{\mathbb{C}_v}^{N+1}$ and $F\in K[x_0,\dots,x_N]$, we may suppose without loss (by passing to an extension if necessary) that $r(F)=0$ in (\ref{eqn:AGfn}) and that given such a lift $F$ with $r(F)=0$, lifts $\widetilde{P_i}\in\mathbb{A}_{\overline{K}}^{N+1}$ of the $P_i$ are chosen so that $\widehat{H}_F(P_i)=0$ for all $i$. We write $X_v=X\times_K\mathbb{C}_v$, and \[\pi_v:\mathbb{A}_{\mathbb{C}_v}^{N+1}\setminus\{(0,\dots,0)\}\to\mathbb{P}_{\mathbb{C}_v}^N\] for the natural projection map. The $v$-adic homogeneous set of $F$ is defined as \[\mathcal{K}_v=\{P\in\mathbb{A}^{N+1}(\mathbb{C}_v):||F^n(P)||_v\not\to\infty\textup{ as }n\to\infty\}.\] \begin{definition}\label{def:dBn} Given a basis $\mathcal{B}_n$ of $H^0(X,\iota^*(\mathcal{O}(n)))$ of order $c(n)$, and $v\in M_K$, let \begin{equation*}d_{\mathcal{B}_n}(\mathcal{K}_v\cap\pi_v^{-1}(X_v))=\exp\left(\max_{\substack{\left(\widetilde{P_1},\dots,\widetilde{P_{c(n)}}\right)\\\in(\mathcal{K}_v\cap\pi_v^{-1}(X_v))^{c(n)}}}\frac{1}{n\cdot c(n)}\log\left|\mathrm{det}\left(\eta_j(\widetilde{P_i})\right)_{\eta_j\in\mathcal{B}_n}\right|_v\right).\end{equation*}  \end{definition} In view of the preceding discussion and Definition \ref{def:dBn}, (\ref{eqn:lowerboundglobal}) becomes \begin{equation}\label{eqn:globalupperbdtransfin}\sum_{v\in M_K}\log d_{\mathcal{B}_n}(\mathcal{K}_v\cap\pi_v^{-1}(X_v))\le\frac{C'\log n}{n}.\end{equation} By the product formula, the left-hand side of (\ref{eqn:globalupperbdtransfin}) is independent of the choice of $\mathcal{B}_n$.  

\subsection{The special basis}

Let $(X,f)$ be a polarized dynamical system as in (\ref{eqn:Fakhruddin}) over a valued field $K$, and let $F=[F_0:\cdots:F_N]$ be a homogeneous lift of $f$ to $\mathbb{A}^{N+1}(K)$. We now form a special spanning set of $H^0(X,\iota^*(\mathcal{O}(n)))$ for each $n$. Any linearly independent subset of this spanning set having cardinality $h^0(X,\iota^*(\mathcal{O}(n)))$ will count as a ``special" basis $\mathcal{B}_n$. Consider the collection $\mathcal{G}$ of polynomials of the form $\left(F_i^{(k)}\right)^j$, where $0\le i\le N$, $k\in\mathbb{Z}_+$, and $1\le j\le d-1$. For $n\ge d(N+1)$ an integer, let $\lfloor n\rfloor_{\mathcal{G}}$ denote the largest element $n'$ of the set \[\mathrm{deg}(\mathcal{G}):=\{\textup{deg}(\Psi):\Psi\in\mathcal{G}\}\] such that \begin{equation}\label{eqn:lmax} (N+1)n'\le n.\end{equation} By the maximality of $\lfloor n\rfloor_{\mathcal{G}}$ subject to the condition (\ref{eqn:lmax}), and the fact that if $n''$ is the smallest element of $\deg(\mathcal{G})$ greater than $n'$, then $n''\le 2n'$, we have that there is an integer $n_0$ such that for all $n\ge n_0$, \begin{equation*}\lfloor n\rfloor_{\mathcal{G}}\ge\frac{1}{2}\cdot\frac{n}{N+1}=\frac{n}{2N+2}.\end{equation*} Hence \begin{equation}\label{eqn:keyratio}  n-\lfloor n\rfloor_{\mathcal{G}}\le\frac{2N+1}{2N+2}n\end{equation} for all $n\ge n_0$.

Now let $n_1\in\mathbb{Z}_+$ be such that $n_1\ge(N+1)d$. By the Fundamental Theorem of Elimination Theory \cite[\S3]{Macaulay}, since $f$ is a morphism and hence $\textup{Res}(F^{(k)})\ne 0$ for all $k$, any homogeneous degree $n_1$ polynomial $\Phi$ in $K[x_0,\dots,x_N]$ may be written as \begin{equation}\label{eqn:Macaulay}\Phi=\eta_0^{(1)}\Psi_0^{(1)}+\dots+\eta_N^{(1)}\Psi_N^{(1)}\end{equation} for $\Psi_i^{(1)}$ homogeneous of degree $\lfloor n_1\rfloor_{\mathcal{G}}$ in $\mathcal{G}$ and $\eta_i^{(1)}$ homogeneous of degree $n_1-\lfloor n_1\rfloor_{\mathcal{G}}$ satisfying \[\frac{N}{N+1}n_1\le n_1-\lfloor n_1\rfloor_{\mathcal{G}}\le\frac{2N+1}{2N+2}n_1,\] where the lower bound follows from (\ref{eqn:lmax}) and the upper bound from (\ref{eqn:keyratio}). Write $n_2=n_1-\lfloor n_1\rfloor_{\mathcal{G}}$. Then provided $n_2\ge(N+1)d$, each $\eta_i^{(1)}$ in turn is of the form \[\eta_i^{(1)}=\eta_0^{(2)}\Psi_0^{(2)}+\dots+\eta_N^{(2)}\Psi_N^{(2)}\] for $\Psi_i^{(2)}$ of degree $\lfloor n_2\rfloor_{\mathcal{G}}$ in $\mathcal{G}$, and $\eta_i^{(2)}$ of degree $n_2-\lfloor n_2\rfloor_{\mathcal{G}}$. Applying (\ref{eqn:keyratio}) again, we have \[\frac{N}{N+1}n_2\le n_2-\lfloor n_2\rfloor_{\mathcal{G}}\le\frac{2N+1}{2N+2} n_2.\] For $n\ge d(N+1)$, let $t_1=\log_{\frac{N+1}{N}}(\max\{1,n-d(N+1)\})$, and let $t_2=\log_{\frac{2N+2}{2N+1}}n$. Continuing in this way, we see that given any $n\ge d(N+1)$, every degree $n$ homogeneous polynomial in $K[x_0,\dots,x_N]$ may be written as a $K$-linear combination of degree $n$ polynomials in the set \begin{equation}\label{eqn:basisform}\mathcal{F}=\bigcup_{j=\lfloor t_1\rfloor}^{\lfloor t_2\rfloor}\left\{\eta_iG_i: G_i\in\mathcal{G}^j\textup{ and }\eta_i\textup{ homog.~of deg.}<d(N+1)\right\}\end{equation} where \[\mathcal{G}^j=\left\{G_1G_2\cdots G_j:G_l\in\mathcal{G}\textup{ for all }1\le l\le j\right\}.\] Thus the degree $n$ elements of $\mathcal{F}$ span $H^0(X,\iota^*(\mathcal{O}(n)))$ over $K$.

\begin{definition}\label{def:Hn}  We will denote by $H(n)$ any basis of $H^0(X,\iota^*(\mathcal{O}(n)))$ consisting of vectors of the form (\ref{eqn:basisform}) whenever $n\ge d(N+1)$, and consisting of the standard monomial basis otherwise. We will denote by $c(n)$ the quantity $h^0(X,\iota^*(\mathcal{O}(n)))$, even though this quantity clearly depends on $(X,f)$ and the diagram (\ref{eqn:Fakhruddin}) in addition to $n$. When there is potential ambiguity as to the underlying polarized dynamical system or the lift $F$ used in defining $H(n)$, we will refer to $H(n)$ and $c(n)$ as being \emph{with respect to $(X,f)$} (as well as the lift $F$). The notation \[\left(\eta_j(P_1,\dots,P_{c(n)})\right)_{\eta_j\in H(n)}\] will denote the $c(n)\times c(n)$ matrix whose $i$-th row corresponds to $P_i$ and whose $j$-th column corresponds to $\eta_j\in H(n)$.\end{definition}

\section{Convergence to the transfinite diameter}\label{section:mainthm}

\subsection{Main theorems}

\begin{thm}\label{thm:basisbound} Let $K$ be a valued field with absolute value $|\cdot|$, and let $(X,f)$ be a degree $d$ polarized dynamical system over $K$ satisfying a commutative diagram \begin{equation*}
\xymatrix{
	{X} \ar[rr]^{\phi} \ar@{^{(}->}[d]^{\iota}
	&& {X} \ar@{^{(}->}[d]^{\iota} \ar@{->}[d]\\
	{\mathbb{P}^N} \ar[rr]^f
	&& {\mathbb{P}^N} 
.}
\end{equation*} Let $F:\mathbb{A}_K^{N+1}\to\mathbb{A}_K^{N+1}$ be a homogeneous lift of $f$, and let $\mathcal{K}$ be its homogeneous filled Julia set. For this $F$, and for $n\in\mathbb{Z}_+$, let $c(n)$ and $H(n)$ be as in Definition \ref{def:Hn}. Let \[\pi:\mathbb{A}^{N+1}\setminus\{(0,\dots,0)\}\to\mathbb{P}^N\] be the natural projection. Let $R$ be the diameter of $\mathcal{K}\cap\pi^{-1}(X)$ with respect to the distance defined by the supnorm. There is a constant $C=C(d,N)$ with the following property. For any $n\ge 2$ and any $c(n)$-tuple $\vec{P}=\left(P_1,\dots,P_{c(n)}\right)$ of points in $\mathcal{K}\cap\pi^{-1}(X)$, we have \begin{equation}\label{eqn:avglb}\log\left|\textup{det}\left(\eta_j(\vec{P})\right)_{\eta_j\in H(n)}\right|\le C\max\{\log R,1\}(\log n)\cdot c(n).\end{equation} (Here we define $\log(0)=-\infty$.)\end{thm}

\begin{proof} Assume without loss that $n\ge d(N+1)$, and write $t_2=\log_{\frac{2N+2}{2N+1}}n$. On $\mathcal{K}$, a polynomial of the form (\ref{eqn:basisform}) has absolute value at most \begin{equation}\label{eqn:Rbound} \max\{R,1\}^{d(N+1)-1}\max\{R,1\}^{\lfloor t_2\rfloor(d-1)}\le \max\{R,1\}^{C'\log n}\end{equation} for some constant $C'=C'(d,N)$. Hadamard's inequality states that if $H$ is a matrix with columns $h_1,\dots,h_m\in K^m$, and $||\cdot||$ denotes the sup-norm if $K$ is non-archimedean and the Euclidean norm otherwise, then \begin{equation}\label{eqn:Hadamard}|\mathrm{det}(H)|\le\prod_{i=1}^m||h_i||.\end{equation} Observing that the matrix $\left(\eta_j(\vec{P})\right)_{\eta_j\in H(n)}$ has $c(n)$ columns, each corresponding to some $\eta_j\in H(n)$, it follows immediately from (\ref{eqn:Rbound}) and (\ref{eqn:Hadamard}) that the left-hand side of (\ref{eqn:avglb}) is at most \[C\max\{\log R,1\}\log(n)\cdot c(n)\] for some constant $C=C(d,N)$. \end{proof} From this we obtain Theorem \ref{thm:mainthm} as an adelic reformulation in the setting of product formula fields.

\begin{definition}\label{def:prodformula} A product formula field is a field $K$ together with a nonempty set $\{|\cdot|_v\}_{v\in M_K}$ of inequivalent nontrivial absolute values $|\cdot|_v$ on $K$ such that:\begin{itemize}\item For each $\alpha\in K^*$, $|\alpha|_v=1$ for all but finitely many $v$. \item Each $\alpha\in K^*$ satisfies the product formula \[\prod_{v\in M_K}|\alpha|_v=1.\]\end{itemize} We note that we do not require the residue field of any completion $K_v$ for $v\in M_K$ to be finite. \end{definition}

\begin{thm}[cf.~Theorem \ref{mainthm:AGhigherdim}]\label{thm:mainthm} Let $K$ be a product formula field with normalized system $~{\{|\cdot|_v\}_{v\in M_K}}$ of absolute values, and let $(X,f)$ be a polarized dynamical system over $K$ satisfying a commutative diagram as in (\ref{eqn:Fakhruddin}). Let $F:\mathbb{A}_K^{N+1}\to\mathbb{A}_K^{N+1}$ be any homogeneous lift of $f$. For each $v\in M_K$, let \[\pi_v:\mathbb{A}_{\mathbb{C}_v}^{N+1}\setminus\{(0,\dots,0)\}\to\mathbb{P}_{\mathbb{C}_v}^N\] be the natural projection, let $X_v=X\times_K\mathbb{C}_v$, and let $\mathcal{K}_v$ be the $v$-adic homogeneous filled Julia set for $F$. For each $n\in\mathbb{Z}_+$, let $\mathcal{B}_n$ be any basis for $H^0(X,\iota^*(\mathcal{O}(n)))$, and let \linebreak$d_{\mathcal{B}_n}(\mathcal{K}_v\cap\pi_v^{-1}(X_v))$ be as in Definition \ref{def:dBn}. Then there is a constant $C=C(f)$ such that for all $n\ge 2$, \begin{equation}\label{eqn:globalupperbd}\sum_{v\in M_K}\log d_{\mathcal{B}_n}(\mathcal{K}_v\cap\pi_v^{-1}(X_v))\le\frac{C\log n}{n}.\end{equation}\end{thm}

\begin{proof} As observed in the discussion surrounding (\ref{eqn:globalupperbdtransfin}), it suffices to prove this in the special case where $\mathcal{B}_n=H(n)$ and $|\cdot|_v$ is the absolute value corresponding to an arbitrary choice of $v\in M_K$. But in that case the statement follows from Theorem \ref{thm:basisbound} applied to $\mathbb{C}_v$. \end{proof}

\begin{cor}\label{cor:transfin} With all notation as in Theorem \ref{thm:mainthm}, suppose that $|\textup{Res}(F)|_v=1$ for all $v\in M_K$, and that for each $n$, the basis $H(n)$ of $H^0(X,\iota^*(\mathcal{O}(n)))$ as in Definition \ref{def:Hn} is defined with respect to this lift $F$. Then for each $v\in M_K$, \begin{equation*}\label{eqn:Hncap}\lim_{n\to\infty} d_{H(n)}(\mathcal{K}_v\cap\pi_v^{-1}(X_v))=1.\end{equation*}  \end{cor} Before beginning the proof of Corollary \ref{cor:transfin}, we introduce the definition of good reduction for $f$ at $v\in M_K$ and a characterization of good reduction in terms of the geometry of the homogeneous filled Julia set $\mathcal{K}_v$ of $f$ at $v$.

\begin{definition*} We call a nonarchimedean $v\in M_K$ a place of \emph{good reduction for $f$}, or alternatively, a \emph{good place of $f$}, if for any homogeneous lift $F=(F_0,\dots,F_N)$ of $f$ such that $|\textup{Res}(F)|_v=1$, one has \[\max\{||F_0||_v,\dots,||F_N||_v\}=1,\] where $||F_i||_v$ is the sup-norm of the coefficients of $F_i$. If $v$ is either archimedean, or is nonarchimedean and is not a good place of $f$, then we say that $f$ has bad reduction at $v$, or alternatively, that $v$ is a bad place of $f$.\end{definition*}

\begin{rmk}\label{rmk:KS} Equivalently, by \cite[Remark 8]{KawaguchiSilverman}, a nonarchimedean $v\in M_K$ is a place of good reduction for $f$ if and only if for some (hence any) homogeneous lift $F$ of $f$, the homogeneous filled Julia set $\mathcal{K}_v$ is a polydisk of polyradius $(r,\dots,r)$ for some $r\in\mathbb{R}_{>0}$. This condition in turn implies that for any $n\ge 1$, for $H(n)$ as in Definition \ref{def:Hn} and $g_{n,v}$ as in (\ref{eqn:AGfn}) for the basis $H(n)$ and the absolute value given by $v$, one has \[g_{n,v}\left(P_1,\dots,P_{c(n)}\right)\ge0\] for any $P_1,\dots,P_{c(n)}$.\end{rmk}

\begin{proof}[Proof of Corollary \ref{cor:transfin}] Let $g_{n,v}$ be given by (\ref{eqn:AGfn}), where $|\cdot|=|\cdot|_v$ and $\mathcal{B}_n=H(n)$. By Remark \ref{rmk:KS}, the fact that $|\textup{Res}(F)|_v=1$ implies that $f$ has good reduction at $v$ if and only if the diameter $R$ of $\mathcal{K}_v$ is equal to $1$. We claim that this forces \begin{equation}\label{eqn:limsup}\limsup_{n\to\infty}d_{H(n)}(\mathcal{K}_v\cap\pi_v^{-1}(X_v))\le 1\end{equation} for each $v\in M_K$. Indeed, by Theorem \ref{thm:basisbound}, for each bad place $v$, there is a constant $C_v$ such that  for all \[\vec{P_v}=\left(P_1,\dots,P_{c(n)}\right)\in(\mathcal{K}_v\cap\pi_v^{-1}(X_v))^{c(n)},\] we have \begin{equation}\label{eqn:badplaces}\log\left|\textup{det}\left(\eta_j(\vec{P_v})\right)_{\eta_j\in H(n)}\right|_v\le C_v \log(n) c(n),\end{equation} and for each good place $v$, (\ref{eqn:badplaces}) holds for all $n$ with $C_v=0$. This proves (\ref{eqn:limsup}). We claim moreover that for each $v\in M_K$, \begin{equation}\label{eqn:liminf} \liminf_{n\to\infty}d_{H(n)}(\mathcal{K}_{v}\cap\pi_{v}^{-1}(X_{v}))\ge 1.\end{equation} Suppose $P_1,\dots,P_{c(n)}\in X(\overline{K})$ are preperiodic points of $X$ such that \[\textup{det}\left(\eta_j(P_1,\dots,P_{c(n)})_{\eta_j\in H(n)}\right)\ne0.\] Then we have \begin{equation*}\label{eqn:greenglobal} \sum_{v\in M_K}g_{n,v}
\left(P_1,\dots,P_{c(n)}\right)=0\end{equation*} by the product formula. Thus if $v'\in M_K$ is such that \[\liminf_{n\to\infty}d_{H(n)}(\mathcal{K}_{v'}\cap\pi_{v'}^{-1}(X_{v'}))<1,\] then there is a sequence of integers $i_1<i_2<\cdots$ and an $\epsilon>0$ independent of $i_n$ such that for all sufficiently large $i_n$, \[\sum_{v\ne v'\in M_K}g_{i_n,v}(P_1,\dots,P_{c(i_n)})<-\epsilon.\] This in turn forces \begin{equation}\label{eqn:sumoverotherv}\sum_{\substack{v\ne v'\in M_K\\v\text{ bad}}}g_{i_n,v}(P_1,\dots,P_{c(i_n)})<-\epsilon\end{equation} as the contribution coming from the places of good reduction is non-negative. But the proof of Theorem \ref{thm:mainthm} in fact demonstrates that there is a $C=C(f)$ such that for all $n\ge2$, \begin{equation*}\label{eqn:strongerthm}\sum_{v\in M_K}\max\{\log d_{H(n)}(\mathcal{K}_v\cap\pi_v^{-1}(X_v)),0\}\le\frac{C\log n}{n},\end{equation*} contradicting (\ref{eqn:sumoverotherv}) for sufficiently large $i_n$. Finally, we note that no matter how large $i_n$ is, by the Zariski density of preperiodic points, we may always find preperiodic points $P_1,\dots,P_{c(i_n)}\in X(\overline{K})$ such that \[\textup{det}\left(\eta_j\left(P_1,\dots,P_{c(i_n)}\right)\right)_{\eta_j\in H(i_n)}\ne0,\] which by the argument just given, proves that \[\liminf_{n\to\infty} d_{H(n)}(\mathcal{K}_v\cap\pi_v^{-1}(X_v))\ge 1\] for each $v\in M_K$. Combining this with (\ref{eqn:limsup}) completes the proof.\end{proof}

\subsection{Equidistribution of asymptotically Fekete tuples}

It is natural to ask about whether the $c(n)$-tuples that realize $d_{H(n)}(\mathcal{K}_v\cap\pi_v^{-1}(X_v))$ equidistribute to a distinguished measure on $X_v^{\textup{an}}$, as they do in the dimension one case. Boucksom and Eriksson \cite{BoucksomEriksson} have shown such a result using pluripotential theory. The differentiability assumption appearing in \cite[Theorem D]{BoucksomEriksson} holds in the case of a smooth variety $X$ over either $\mathbb{C}$ or a complete discretely or trivially valued field of residue characteristic $0$, among other fields. In a subsequent paper by Boucksom--Gubler--Martin \cite[Theorem A]{BoucksomGublerMartin}, it was shown that this differentiability assumption holds in the case of a continuous \emph{psh} metric over an arbitrary complete nonarchimedean field. In this section, we assume that $K$ is a complete nontrivially valued nonarchimedean field or that $K=\mathbb{C}$. Here we relate \cite[Theorem D]{BoucksomEriksson} and \cite[Corollary D]{BoucksomGublerMartin} to the Arakelov-Green's functions $g_n$, and note how the proof of \cite[Theorem D]{BoucksomEriksson} in fact yields a stronger theorem. This results in the following application.

\begin{thm}\label{thm:eqmsreuniqueness} Let $(X,f)$ be a polarized dynamical system over $K$, where $K$ is either $\mathbb{C}$ or a complete nontrivially valued nonarchimedean field, given by $\iota:X\to\mathbb{P}^N$ and $f:\mathbb{P}^N\to\mathbb{P}^N$ satisfying a commutative diagram as in (\ref{eqn:Fakhruddin}). If $K=\mathbb{C}$, assume further that $X$ is normal. Let $\pi:\mathbb{A}^{N+1}\setminus\{(0,\dots,0)\}\to\mathbb{P}^N$ be the natural projection. Then there is a unique probability measure  $\mu=\mu(X,f)$ on $X^{\textup{an}}$ with the following property. Fix any homogeneous lift $F$ of $f$, with homogeneous filled Julia set $\mathcal{K}$. Let $\{\mathcal{B}_{n}\}_{n=1}^\infty$ be any sequence of bases of $H^0(X,\iota^*(\mathcal{O}(n)))$ such that the limit \[\lim_{n\to\infty}d_{\mathcal{B}_{n}}(\mathcal{K}\cap \pi^{-1}(X))\] exists and is nonzero. Let $\{\vec{P_{n}}\}_{n=1}^\infty$ be any sequence of points of $(\mathcal{K}\cap\pi^{-1}(X))^{c(n)}$ respectively such that \begin{equation}\label{eqn:asympFekete}\lim_{n\to\infty}\left|\textup{det}\left(\eta_j(\vec{P_{n}})\right)_{\eta_j\in\mathcal{B}_{n}}\right|^{1/(n\cdot c(n))}=\lim_{n\to\infty}d_{\mathcal{B}_{n}}(\mathcal{K}\cap\pi^{-1}(X)).\end{equation} For each $n$ and each $1\le i\le c(n)$, let \[\pi_{i,n}:(\mathcal{K}\cap\pi^{-1}(X))^{c(n)}\to\mathcal{K}\cap\pi^{-1}(X)\] denote the projection to the $i$th coordinate. Then on $X^{\textup{an}}$, the sequence of probability measures \begin{equation}\label{eqn:nu}\nu_{n}=\frac{1}{c(n)}\sum_{\substack{1\le i\le c(n)\\x=\pi_{i,n}\left(\vec{P_{n}}\right)}}\delta_x\end{equation} equidistributes to $\mu$. \end{thm} 

\begin{proof} It suffices to recall the notation used by Boucksom and Eriksson and to relate it to our setting. For $X$ a geometrically reduced projective scheme over $K$ of dimension $N$, $L$ a line bundle on $X$, and $\phi,\psi$ metrics on $L$, Boucksom and Eriksson construct a transfinite diameter $\delta_\infty(\phi,\psi)$ associated to $\phi$ and $\psi$. The construction is as follows. Let $M=\textup{dim}\,H^0(X,L)$. A basis $\vec{\textbf{s}}=(s_i)$ of $H^0(X,L)$ determines a generator $s_1\wedge\cdots\wedge s_M$ of $\mathrm{det}\,H^0(X,L)$. Let $L^{\boxtimes{M}}$ denote the $M$-th external tensor power of $L$, i.e., the line bundle \begin{equation}\label{eqn:externaltensorprod}\mathrm{pr}_1^*(L)\otimes\cdots\otimes\mathrm{pr}_M^*(L)\end{equation} on $X^M$, where $\mathrm{pr}_i$ is the $i$-th projection map from $X^M$ to $X$. There is a natural map $\Upsilon:\mathrm{det}\,H^0(X,L)\to H^0(X^M,L^{\boxtimes M})$ given by post-composing the map \begin{align}\begin{split}\label{eqn:detmap1}\mathrm{det}\,H^0(X,L)&\to H^0(X,L)^{\otimes M}\\ \mathbf{s}=s_1\wedge\cdots\wedge s_M&\mapsto\sum_{\sigma\in S_M}(-1)^{\mathrm{sign}(\sigma)}s_{\sigma(1)}\otimes\cdots\otimes s_{\sigma(M)}\end{split}\end{align} with the canonical isomorphism $H^0(X,L)^{\otimes M}\congto H^0(X^M,L^{\boxtimes M})$. In other words, we have \begin{equation}\label{eqn:detmap2}\Upsilon(s_1\wedge\cdots\wedge s_M)(P_1,\dots,P_M)=\mathrm{det}(s_j(P_i))_{1\le i,j\le M}.\end{equation} We will thus suggestively write \[\mathrm{det}\,\mathbf{s}:=\Upsilon(\mathbf{s})\in H^0(X^M,L^{\boxtimes M}).\] Every continuous metric $\phi$ on $L$ induces a continuous metric $\phi^{\boxtimes M}$ on $L^{\boxtimes M}$ via the tensor product metric on (\ref{eqn:externaltensorprod}). 
	
For a metric $\phi$ on $L$, let $|s|_\phi(P)$ denote the image of $(s,P)$ under the associated map $H^0(X,L)\times X(K)\to\mathbb{R}$, and likewise for the induced metric $\phi^{\boxtimes M}$ on $L^{\boxtimes M}$. Write \[||\mathrm{det}\,\textbf{s}||_{\phi^{\boxtimes M}}=\sup_{P\in(X^M)^{\mathrm{an}}}|\mathrm{det}\,\textbf{s}|_{\phi^{\boxtimes M}}(P).\] A \emph{Fekete configuration} for the metric $\phi$ is a point $P\in\left(X^M\right)^{\textup{an}}$ such that \[|(\mathrm{det}\,\textbf{s})|_{\phi^{\boxtimes M}}(P)=||\mathrm{det}\,\textbf{s}||_{\phi^{\boxtimes M}}.\] (Here we extend $|\mathrm{det}\,\textbf{s}|_{\phi^{\boxtimes M}}$ uniquely to a continuous function on $(X^M)^{\mathrm{an}}$.) Similarly, any metric $\psi$ on $L$ induces a supnorm $||\cdot||_\psi$ on $H^0(X,L)$ via \[||s||_\psi:=\sup_{P\in X^{\textup{an}}}|s|_\psi(P)\] for each $s\in H^0(X,L)$. Finally, $||\cdot||_\psi$ induces a norm $\mathrm{det}||\cdot||_\psi$ on $\mathrm{det}\,H^0(X,L)$ by setting \[\mathrm{det}||\tau||_\psi=\inf_{\tau=v_1\wedge\cdots\wedge v_M}\prod_i||v_i||_\psi.\] For $\vec{\textbf{s}}=(s_i)$ a basis of $H^0(X,L)$, $\mathrm{det}\,\mathbf{s}$ the associated element of $H^0(X^M,L^{\boxtimes M})$ where $\mathbf{s}=s_1\wedge\cdots\wedge s_M$, and $\phi,\psi$ continuous metrics on $L$ as above, write \[\delta(\phi,||\cdot||_{\psi})=\frac{||\mathrm{det}\,\textbf{s}||_{\phi^{\boxtimes M}}}{\mathrm{det}||\mathbf{s}||_\psi}.\] (This quotient is clearly independent of the choice of basis $\vec{\textbf{s}}$.) Boucksom and Eriksson \cite[Theorem C]{BoucksomEriksson} prove that the limit \begin{equation}\label{eqn:limitC}\delta_\infty(\phi,\psi):=\lim_{m\to\infty}\delta\left(m\phi,||\cdot||_{m\psi}\right)^{N!/m^{N+1}}\end{equation} exists in $\mathbb{R}_{>0}$, and equals the so-called relative volume $\mathrm{vol}(L,\psi,\phi)$ of $\psi$ and $\phi$. Furthermore, combining \cite[Theorem A]{BoucksomGublerMartin} with \cite[Theorem D/Theorem 10.10]{BoucksomEriksson} shows that if: \begin{itemize} \item for each $m\ge 1$, $\vec{P_m}$ is a Fekete configuration for $m\phi$ (i.e., a $h^0(X,L^m)$-tuple realizing the supnorm $||\mathrm{det}\,\textbf{s}||_\phi$ for some (hence any) $\textbf{s}\in\mathrm{det}\,H^0(X,L^m))$, and \item $\phi$ is a uniform limit of nef model metrics (a ``\emph{psh} metric"),\end{itemize} then the sequence $\{\vec{P_m}\}_{m=1}^\infty$ equidistributes to a canonical probability measure $\mu$ on $X^{\textup{an}}$, in the sense of (\ref{eqn:nu}) \cite[Theorem D/Theorem 10.10]{BoucksomEriksson}. (We note that when $K=\mathbb{C}$, we also require $X$ to be normal; see \cite[Theorem B(ii) and p.12]{BoucksomGublerMartin}.) In fact, their proof demonstrates even more. For metrics $\phi$ and $\psi$ on $L$, $\mathbf{s}\in\mathrm{det}\,H^0(X,L)$, and $P\in\left(X^M\right)^{\textup{an}}$, write \[V_{\phi,||\cdot||_\psi}(P):=\frac{|\mathrm{det}\,\mathbf{s}|_{\phi^{\boxtimes M}}(P)}{\mathrm{det}||\mathbf{s}||_\psi},\] so that \[\delta(\phi,||\cdot||_\psi)=\sup_{P\in\left(X^M\right)^{\textup{an}}}V_{\phi,||\cdot||_\psi}(P).\] We call a sequence $\{\vec{P_m}\}_{m=1}^\infty$ \emph{asymptotically Fekete} for $\phi$ if for some (hence any) metric $\psi$ on $L$,\[\lim_{m\to\infty}\left(V_{m\phi,||\cdot||_{m\psi}}(\vec{P_m})\right)^{N!/m^{N+1}}=d_\infty(\phi,\psi).\] Combining \cite[Theorem A]{BoucksomGublerMartin} with \cite[proof of Theorem 10.10]{BoucksomEriksson} shows that if: \begin{itemize} \item $\{\vec{P_m}\}_{m=1}^\infty$ is asymptotically Fekete for $\phi$, and \item $\phi$ is a uniform limit of nef model metrics,\end{itemize} then $\{\vec{P_m}\}_{m=1}^\infty$ equidistributes to the aforementioned distinguished probability measure $\mu$ on $X^{\textup{an}}$.

To apply this theorem to our polarized dynamical system $(X,f)$, realized as $\iota:X\hookrightarrow\mathbb{P}^r$ and $f:\mathbb{P}^r\to\mathbb{P}^r$ for some $r$, let $K$ be as above with absolute value $|\cdot|$, let $F$ be a homogeneous lift of $f$ with $|\textup{Res}(F)|=1$ and homogeneous filled Julia set $\mathcal{K}$, let $H(m)$ be associated to $X$ and $F$ as in Definition \ref{def:Hn}, and let $\phi$ be the unique (up to scaling) $f$-invariant continuous metric on $L=\iota^*(\mathcal{O}(m))$, given explicitly by \begin{equation}\label{eqn:canonicalmetric}|s|_{\phi}(P)=\frac{|s(\widetilde{P})|}{\textup{exp}(\widehat{H}_F(\widetilde{P}))}\end{equation} for $s\in H^0(X,L)$ and $\widetilde{P}$ any lift to $\mathbb{A}^{r+1}(K)$ of $P\in X(K)$. By the explicit construction of $\phi$ given in \cite[Proof of Theorem 2.2]{Zhang:smallpoints}, and by the fact that pullbacks of model metrics are model metrics \cite[Proposition 5.13]{BoucksomEriksson}, we have that $\phi$ is a uniform limit of (nef) model metrics, i.e., the metric $\phi$ is \emph{psh} in the language of \cite{BoucksomEriksson,BoucksomGublerMartin}.

By (\ref{eqn:detmap1}) and (\ref{eqn:detmap2}), it follows that for $\textbf{s}\in\mathrm{det}\,H^0(X,L)$ and $\mathrm{det}\,\textbf{s}$ the corresponding section of $H^0(X^M,L^{\boxtimes M})$, where $M=\mathrm{dim}\,H^0(X,L^m)$, we have \[||\mathrm{det}\,\textbf{s}||_{\phi^{\boxtimes M}}\left((P_1,\dots,P_M)\right)=\frac{|\mathrm{det}(\textbf{s}(P_1,\dots,P_M))|}{\prod_{i=1}^M\mathrm{exp}\left(\widehat{H}_F(P_i)\right)^m}.\] Therefore, for $\vec{P_m}=(P_1,\dots,P_M)$ such that $\mathrm{det}(\textbf{s}(P_1,\dots,P_M))\ne0$, \[\log||\mathrm{det}\,\textbf{s}||_{\phi^{\boxtimes M}}((P_1,\dots,P_M))=-Mm\cdot g_m(P_1,\dots,P_M).\] It follows from the aforementioned result of Boucksom and Eriksson that there is a unique probability measure $\mu$ on $X^{\textup{an}}$ with the following property: for any sequence of bases $\mathcal{B}_m$ of $H^0(X,\iota^*(\mathcal{O}(m)))$ such that \[\lim_{m\to\infty}d_{\mathcal{B}_m}(\mathcal{K}\cap\pi^{-1}(X))\] exists and is nonzero, and for any asymptotically Fekete sequence as in (\ref{eqn:asympFekete}), the sequence of probability measures (\ref{eqn:nu}) equidistributes to $\mu$. \end{proof}

\section{A Lehmer-type lower bound on the canonical height for abelian varieties}\label{section:Lehmer}

In this section, we prove a Lehmer-style lower bound on the canonical height of abelian varieties over $K$ a product formula field with perfect residue fields (Theorem \ref{thm:Lehmerbdabvar}). The method presented below combines aspects of \cite[proof of Theorem 1.14]{Baker} and \cite{LooperSilverman}. In the proof of Theorem \ref{thm:Lehmerbdabvar}, an arithmetic progression (in the translation-by-$P$ map for $P\in A(K')$ a non-torsion point of $A$) is argued to lie in a specific subset $U$ of $A(\mathbb{C}_{w})$ for a place $w$ of $K'$ above the given place $v_0$ of bad reduction. Extending the duplication map on $A$ to an endomorphism $f:\mathbb{P}^N\to\mathbb{P}^N$ as in (\ref{eqn:Fakhruddin}), and applying Theorem \ref{thm:eqmsreuniqueness}, one can show that there is a sequence of positive integers $n_1<n_2<\dots$ such that \[d_{H(n_l)}(\mathcal{K}_{w}\cap\pi_{w}^{-1}(U))<d_{H(n_l)}(\mathcal{K}_{w}\cap\pi_{w}^{-1}(A(\mathbb{C}_w)))\] for all sufficiently large $l$, where $H(n_l)$ is a basis as in Definition \ref{def:Hn} with respect to any lift $F$ of $f$ with homogeneous filled Julia set $\mathcal{K}_{w}$.

\begin{lem}\label{lem:suitablemultiples} Let $K$ be a field, and let $A$ be a dimension $g$ geometrically simple abelian variety over $K$. Let $\mathcal{L}$ be a very ample line bundle on $A$, and let $\iota:A\hookrightarrow\mathbb{P}(V)$ be an embedding induced by $\mathcal{L}$, where $V=H^0(A,\mathcal{L})$. For $n\in\mathbb{Z}_+$, let $c(n)=\mathrm{h}^0(A,\iota^*(\mathcal{O}(n)))$, and let $\mathcal{B}_n$ be any basis for $H^0(A,\iota^*(\mathcal{O}(n)))$. Let $P\in A(\overline{K})\setminus A(\overline{K})_{\textup{tors}}$. Then the set \[\{P,2P,\dots,(2n^g+c(n))P\}\] contains a $c(n)$-element subset $\{P_1,\dots,P_{c(n)}\}$ such that for $\vec{P}=\left(P_1,\dots,P_{c(n)}\right)$, we have \[\textup{det}(\eta_j(\vec{P}))_{\eta_j\in\mathcal{B}_n}\ne 0.\] \end{lem}

\begin{proof} Let $\mu:A\to A$ be translation by the non-torsion point $P\in A(\overline{K})$, and let $\mathcal{Z}$ be an algebraic cycle on $A$ of degree $n$, where in general the degree of $\mathcal{Z}$ is defined as the degree of $\iota(\mathcal{Z})$ in $\mathbb{P}(V)$, and the degree of $\iota(\mathcal{Z})$ is the sum of the degrees of its irreducible components. The automorphism $\mu^*:H^0(A,\mathcal{L})\to H^0(A,\mathcal{L})$ induces an automorphism $\psi:\mathbb{P}(V)\to\mathbb{P}(V)$, satisfying $\psi\circ\iota=\iota\circ\mu$. Hence $\mu$ extends to the map $\psi$ on $\mathbb{P}(V)$. By abuse of notation, we will identify $\mu$ with $\psi$ in what follows.

\textbf{Claim}: If $\mathcal{Y}$ is of the form \[\mathcal{Y}=\mathcal{Z}\cap\mu(\mathcal{Z})\cap\cdots\cap\mu^l(\mathcal{Z})\] for $l\in\mathbb{Z}_+$ then \[\deg(\mathcal{Y})\le n^{\mathrm{codim}(\mathcal{Y})}.\] (Here, $\mathcal{Y}$ is not necessarily pure-dimensional, and $\mathrm{codim}(\mathcal{Y})$ is defined to be the \emph{maximal} codimension in $A$ among all irreducible components of $\mathcal{Y}$.) 

Proof of claim: We induct on the codimension. Clearly the claim holds when $\mathrm{codim}(\mathcal{Y})=1$. Suppose the claim holds whenever the codimension is at most $\delta-1$ for $\delta\ge 2$. We prove that the claim holds when the codimension is $\delta$. Let \[\mathcal{Y}'=\mathcal{Z}\cap\mu(\mathcal{Z})\cap\cdots\cap\mu^{l-1}(\mathcal{Z})\] be such that $\mathrm{codim}(\mathcal{Y}')<\delta$ and $\mathrm{codim}(\mathcal{Y}'\cap\mu^l(\mathcal{Z}))=\delta$. Let $\mathcal{W}_i$ be an irreducible component of $\mathcal{Y}'$. If $\mathcal{W}_i\cap\mu^l(\mathcal{Z})=\mathcal{W}_i$, then clearly $\deg(\mathcal{W}_i\cap\mu^l(\mathcal{Z}))=\deg(\mathcal{W}_i)$. If $\mathcal{W}_i\cap\mu^l(\mathcal{Z})\ne\mathcal{W}_i$, then \[\deg(\mathcal{W}_i\cap\mu^l(\mathcal{Z}))\le\deg\mu^l(\mathcal{Z})\cdot\deg(\mathcal{W}_i)=\deg(\mathcal{Z})\cdot\deg(\mathcal{W}_i).\] Writing $\mathcal{Y}=\mathcal{Y}'\cap\mu^l(\mathcal{Z})$, it follows that \[\deg(\mathcal{Y})\le\deg(\mathcal{Z})\cdot\sum_{i\in\mathcal{I}}\deg(\mathcal{W}_i)=\deg(\mathcal{Z})\cdot\deg(\mathcal{Y}'),\] where $\mathcal{I}$ is an indexing set for the irreducible components of $\mathcal{Y}'$. But by the induction hypothesis, $\deg(\mathcal{Y}')\le n^{\mathrm{codim}(\mathcal{Y}')}\le n^{\delta-1}$, so \[\deg(\mathcal{Y})\le n^{\delta-1}\cdot\mathrm{deg}(\mathcal{Z})=n^{\delta}=n^{\mathrm{codim}(\mathcal{Y})}.\] This proves the claim.

Returning to the proof of Lemma \ref{lem:suitablemultiples}, let $\mathcal{Y}$ as above have dimension $g_0$ (where we are not assuming $\mathcal{Y}$ is pure-dimensional). Let $\mathcal{Y}_i$ be an irreducible component of $\mathcal{Y}$ of dimension $g_0$. Since $\mathcal{Y}_i$ has maximal dimension among the irreducible components of $\mathcal{Y}$, either \begin{equation}\label{eqn:star}\mu(\mathcal{Y}_j)=\mathcal{Y}_i\end{equation} for some irreducible component $\mathcal{Y}_j$ of $\mathcal{Y}$, or \begin{equation*}\label{eqn:starprime}\dim(\mu(\mathcal{Y})\cap\mathcal{Y}_i)<g_0.\end{equation*} Note that (\ref{eqn:star}) cannot occur with $i=j$, as $P$ is non-torsion. Indeed, let \[\textup{Stab}(\mathcal{Y}_i)=\{P\in A(\overline{K}):P+\mathcal{Y}_i=\mathcal{Y}_i\}.\] Since $\textup{Stab}(\mathcal{Y}_i)$ is a closed subgroup of $A_{\overline{K}}$, and hence an irreducible cycle in $\textup{Stab}(\mathcal{Y}_i)$ of maximal dimension must be preperiodic under $f$, it follows from the fact that $A_{\overline{K}}$ is simple that $\textup{Stab}(\mathcal{Y}_i)$ is a finite set of torsion points. In particular, as $P$ is non-torsion, we deduce that $P\notin\textup{Stab}(\mathcal{Y}_i)$, and hence $\mu(\mathcal{Y}_i)\ne\mathcal{Y}_i$. A similar argument applied to the iterates $\mu^k$ of $\mu$ shows that $\mu^k(\mathcal{Y}_j)=\mathcal{Y}_i$ forces $i\ne j$. 

From this we deduce that if the number of irreducible components of $\mathcal{Y}$ is $r$, then \[\dim(\mu^s(\mathcal{Y})\cap\mathcal{Y}_i)<g_0\] for some $s\le r$. A fortiori, \[\dim(\mathcal{Y}\cap\mu(\mathcal{Y})\cap\cdots\cap\mu^r(\mathcal{Y})\cap\mathcal{Y}_i)<g_0.\]As $\mathcal{Y}_i$ was an arbitrary component of $\mathcal{Y}$ of maximal dimension $g_0$, this gives \[\dim(\mathcal{Y}\cap\mu(\mathcal{Y})\cap\cdots\cap\mu^r(\mathcal{Y}))<g_0,\] which, since $r\le\deg(\mathcal{Y})$, in turn yields \[\dim(\mathcal{Y}\cap\mu(\mathcal{Y})\cap\cdots\cap\mu^{\deg(\mathcal{Y})}(\mathcal{Y}))<g_0.\] By the Claim, we have \[\deg(\mathcal{Y})\le n^{\mathrm{codim}(\mathcal{Y})},\] so \[\dim(\mathcal{Y}\cap\mu(\mathcal{Y})\cap\cdots\cap\mu^{n^{\mathrm{codim}(\mathcal{Y})}}(\mathcal{Y}))<g_0.\] We conclude that \begin{equation}\label{eqn:dimcut}\dim(\mathcal{Z}\cap\mu(\mathcal{Z})\cap\cdots\cap\mu^{n'}(\mathcal{Z}))=-1\end{equation} for $n'=n+n^2+\dots+n^g\le 2n^g$, where $g=\dim(A)$. 

Finally, we apply (\ref{eqn:dimcut}) to a degree $n$ hypersurface $\mathcal{Z}\subset A$ containing all points $Q\in A(\overline{K})$ with the property that for any choice of $z_3,\dots,z_{c(n)}\in A(\overline{K})$, \[\textup{det}\left(\eta_j\left(P,Q,z_3,\dots,z_{c(n)}\right)\right)_{\eta_j\in\mathcal{B}(n)}=0,\] where $P\in A(\overline{K})\setminus A(\overline{K})_{\text{tors}}$ is the fixed non-torsion point in the statement of the lemma. (Note that such a hypersurface $\mathcal{Z}$ must exist, for otherwise, $\mathcal{B}_n$ is not linearly independent on $A$.) Equation (\ref{eqn:dimcut}) applied to this $\mathcal{Z}$ says exactly that among the list \begin{equation}\label{eqn:list} P,2P,3P,\dots,(2n^g)P,\end{equation} there must be at least one element $kP$ such that $kP\notin\mathcal{Z}$. Applying this fact repeatedly (i.e., in the second iteration, choosing a $kP\notin\mathcal{Z}$ among (\ref{eqn:list}) and considering a degree $n$ hypersurface $\mathcal{Z}'$ given by the locus of points $Q$ such that for any $z_4,\dots,z_{c(n)}$, \[\textup{det}\left(\eta_j(P,kP,Q,z_4,\dots,z_{c(n)})\right)_{\eta_j\in\mathcal{B}_n}=0),\] we see that among the list \begin{equation*} P,2P,3P,\dots,(\max\{2n^g,c(n)\})P,\end{equation*} there must be a sublist of $c(n)$ elements $P_1,\dots,P_{c(n)}$ such that if \[\vec{P}=\left(P_1,\dots,P_{c(n)}\right),\] then \[\mathrm{det}\left(\eta_j(\vec{P})\right)_{\eta_j\in\mathcal{B}_n}\ne0.\]\end{proof}

\begin{definition*} Let $A/K$ be an abelian variety over a product formula field $K$. We call a nonarchimedean place $v\in M_K$ a place of \emph{bad reduction of} $A$ if there is no finite extension $L/K$ such that $A\times_KL$ has potential good reduction over $L$. We say that $A$ has \emph{everywhere potential good reduction} if no nonarchimedean $v\in M_K$ is a place of bad reduction of $A$.\end{definition*}

\begin{thm}\label{thm:Lehmerbdabvar} Let $K$ be a product formula field having perfect residue fields at its completions, and let $A/K$ be an abelian variety. Fix a symmetric ample line bundle $\mathcal{L}$ on $A$ and let $\hat{h}_\mathcal{L}$ be the associated canonical height function. There is a constant $C=C(A,K,\mathcal{L})>0$ such that for any $P\in A(\overline{K})\setminus A(\overline{K})_{\textup{tors}}$ that does not lie in any torsion translate of an abelian subvariety of $A\times_K L$ having everywhere potential good reduction for some finite extension $L/K$, we have \[\hat{h}_\mathcal{L}(P)\ge\frac{C}{D^{2g+3}(\log D)^{2g}}\] whenever $D:=[K(P):K]\ge2$. \end{thm}

\begin{proof} Let $A$ and $P$ be as in the statement of the theorem, and suppose without loss of generality that $A$ is geometrically simple of dimension $g$. Let $\phi$ be the duplication map on $A$, so that $\phi^*(\mathcal{L})\simeq\mathcal{L}^4$. By \cite[proof of Corollary 2.2]{Fakhruddin}, there is an $m\in\mathbb{Z}_+$ such that $\mathcal{L}^m$ is very ample and the map $\iota:A\to\mathbb{P}(H^0(A,\mathcal{L}^m))$ given by evaluating each point in $A(\overline{K})$ on a fixed basis of $H^0(A,\mathcal{L}^m)$ yields a commutative diagram (\ref{eqn:Fakhruddin}), where $f$ is a morphism extending $\phi$ and $N=\mathrm{dim}(H^0(A,\mathcal{L}^m))-1$. Assume that $m\in\mathbb{Z}_+$ is chosen to be minimal with this property. For the endomorphism $f:\mathbb{P}^N\to\mathbb{P}^N$ in the induced diagram (\ref{eqn:Fakhruddin}), we have \begin{equation}\label{eqn:htrelation}\hat{h}_{\mathcal{L}}(P)=\frac{1}{m}\hat{h}_f(\iota(P))\end{equation} by (\ref{eqn:htmult}). We will identify points of $A$ with their images under $\iota$.

Let $F$ be any homogeneous lift of $f$ such that $|\textup{Res}(F)|_v=1$ for all $v\in M_K$. (We may have to pass to an extension to achieve this, which we allow without loss of generality.) For each $n\in\mathbb{Z}_+$, let $H(n)$ be defined as in Definition \ref{def:Hn} with respect to $(A,f)$ and the lift $F$, and let $c(n)=h^0(A,\iota^*(\mathcal{O}(n)))$. For each $v\in M_K$, let $\mathcal{K}_v$ be the homogeneous filled Julia set of $F$, let $\pi_v:\mathbb{A}_{\mathbb{C}_v}^{N+1}\setminus\{(0,\dots,0)\}\to\mathbb{P}_{\mathbb{C}_v}^N$ be the natural projection, and let $A_v=A\times_K\mathbb{C}_v$. Then Corollary \ref{cor:transfin} says that for each $v\in M_K$, \begin{equation}\label{eqn:corcapHn}\lim_{n\to\infty}d_{H(n)}(\mathcal{K}_v\cap\pi_v^{-1}(A_v)))=1.\end{equation} Letting \[g_{n,v}\left(P_1,\dots,P_{c(n)}\right)=\frac{1}{c(n)}\sum_{i=1}^{c(n)}\widehat{H}_{F,v}(\widetilde{P_i})-\frac{1}{n\cdot c(n)}\log\left|\textup{det}\left(\eta_j(\widetilde{P_i})_{\eta_j\in H(n)}\right)\right|_v\] for any lifts $\widetilde{P_1},\dots,\widetilde{P_{c(n)}}$ of $P_1,\dots,P_{c(n)}\in A(\mathbb{C}_v)$ respectively, we further have that by Theorem \ref{thm:basisbound}, there is a $C_1'=C_1'(A,f)>0$ such that for each $v\in M_K$, for any $n\ge 2$ and any $P_1,\dots,P_{c(n)}\in A(\mathbb{C}_v)$, \begin{equation}\label{eqn:energylb} g_{n,v}\left(P_1,\dots,P_{c(n)}\right)\ge-\frac{C_1'\log n}{n}.\end{equation} If $v\in M_K$ is a place of good reduction for $f$, and if $P_1,\dots,P_{c(n)}\in A(\mathbb{C}_v)$ are such that $\textup{det}\left(\eta_j(\widetilde{P_i})_{\eta_j\in H(n)}\right)\ne0$, then by Remark \ref{rmk:KS}, \begin{equation}\label{eqn:goodred}g_{n,v}\left(P_1,\dots,P_{c(n)}\right)\ge0.\end{equation} By (\ref{eqn:energylb}) and (\ref{eqn:goodred}), there is a constant $C_1=C_1(A,f)$ such that for all $n\ge 2$, \begin{equation}\label{C1}\sum_{v\in M_K}\min\left\{\min_{P_1,\dots,P_{c(n)}\in A(\mathbb{C}_v)}g_{n,v}\left(P_1,\dots,P_{c(n)}\right),0\right\}\ge-\frac{C_1\log n}{n}.\end{equation} For each $v\in M_K$, fix an embedding of $\overline{K_v}$ into $\mathbb{C}_v$. If for each $n\ge 2$ we let $E_n=\{R_1,\dots,R_{c(n)}\}$ be a set of $c(n)$ torsion points of $A(\overline{K})$ such that for some (hence any) lifts $\widetilde{R_i}\in\mathbb{A}^{N+1}(\overline{K})$ of the $R_i$, \[\textup{det}\left(\eta_j(\widetilde{R_i})\right)_{\eta_j\in H(n)}\ne0\] (which we may achieve by the Zariski density of the set of torsion points), then by (\ref{eqn:goodred}), \[g_{n,v}\left(R_1,\dots,R_{c(n)}\right)\ge 0\] for all good places $v$ of $f$. Moreover, by (\ref{C1}), there is a $C_1'=C_1'(A,f)$ such that \[g_{n,v}\left(R_1,\dots,R_{c(n)}\right)\ge-\frac{C_1'\log n}{n}\] for all bad places $v$ of $f$. Since the $R_i$ are torsion, the product formula thus forces, for each $v\in M_K$, \begin{equation}\label{eqn:torsionlimit}\lim_{n\to\infty}\left|\textup{det}\left(\eta_j(\widetilde{R_i})\right)_{\eta_j\in H(n)}\right|_v^{1/(n\cdot c(n))}=1\end{equation} for any lifts $\widetilde{R_i}\in\mathbb{A}^{N+1}(\mathbb{C}_v)$ of the $R_i$ such that $\widetilde{R_i}\in\mathcal{K}_v$. In particular, by (\ref{eqn:corcapHn}), for each $v\in M_K$, any sequence of such sets \[E_n=\{R_1,\dots,R_{c(n)}\}\subset A(\overline{K})_{\textup{tors}}\] equidistributes to the distinguished measure on $A_v^{\textup{an}}$ that is the subject of Theorem \ref{thm:eqmsreuniqueness}.

On the other hand, let $v_0$ be a place of bad reduction of $A$ (which exists by our assumption on $P$). The sets of $n$-torsion points of $A(\overline{K})$ equidistribute to the so-called canonical measure $\mu_{\textup{can}}$ on $A_{v_0}^{\textup{an}}$ \cite[Theorem 3.1]{ChambertLoir}. Gubler showed that $\mu_{\textup{can}}$ is equal to the Haar measure on the skeleton of $A_{v_0}^{\textup{an}}$ \cite{Gubler}, which is topologically a real torus. As $v_0$ is a place of bad reduction of $A$, the skeleton is positive-dimensional. We also have that any two points of $A(K_{v_0})$ lying in the identity component of the N\'{e}ron model $\mathcal{A}_{v_0}$ of $A\times_KK_{v_0}$ retract to the same point $\xi$ of the skeleton $\Sigma_{v_0}$ of $A_{v_0}^{\textup{an}}$ (cf.~\cite[pp.~13-14]{DeJongShokrieh}). Let \[U=\{P\in\pi_{v_0}(\mathcal{K}_{v_0})\cap A(\mathbb{C}_{v_0}): P\textup{ retracts to }\xi\}.\] We claim that Theorem \ref{thm:eqmsreuniqueness} together with (\ref{eqn:corcapHn}) and (\ref{eqn:torsionlimit}) implies that there is a sequence $\{n_l\}_{l=1}^\infty$ of integers $n_1<n_2<\dots$ such that \begin{equation}\label{eqn:nj}\lim_{l\to\infty} d_{H(n_l)}(\mathcal{K}_{v_0}\cap\pi_{v_0}^{-1}(U))\end{equation} exists and satisfies \begin{equation}\label{eqn:inequality}\lim_{l\to\infty} d_{H(n_l)}(\mathcal{K}_{v_0}\cap\pi_{v_0}^{-1}(U))<\lim_{l\to\infty}d_{H(n_l)}(\mathcal{K}_{v_0}\cap\pi_{v_0}^{-1}(A_{v_0}))=\lim_{n\to\infty}d_{H(n)}(\mathcal{K}_{v_0}\cap\pi_{v_0}^{-1}(A_{v_0})).\end{equation} Indeed, by (\ref{eqn:corcapHn}), (\ref{eqn:torsionlimit}), and the preceding discussion surrounding $\mu_{\mathrm{can}}$, it follows from Theorem \ref{thm:eqmsreuniqueness} that any sequence $\{E_n\}_{n=1}^\infty$ of sets $E_n=\{E_{n,1},\dots,E_{n,c(n)}\}\subseteq A(\mathbb{C}_{v_0})$ with $|E_n|=c(n)$ that satisfies \begin{equation}\label{eqn:liftlim}\lim_{n\to\infty}\left|\mathrm{det}\left(\eta_j(\widetilde{E_{n,i}})\right)_{\eta_j\in H(n)}\right|_{v_0}^{1/(n\cdot c(n))}=1\end{equation} for some choice of lifts $\widetilde{E_{n,i}}\in\mathcal{K}_{v_0}$ of $E_{n,i}$ must equidistribute to the Haar measure on the skeleton of $A_{v_0}^{\textup{an}}$. For a subsequential limit of $\{d_{H(n)}(\mathcal{K}_{v_0}\cap\pi_{v_0}^{-1}(U))\}_{n=1}^\infty$ given by the indices $\{n_l\}_{l=1}^\infty$ (such a subsequential limit must exist by the compactness of the interval $[0,R]\subset\mathbb{R}$ for $R$ the supnorm diameter of $\mathcal{K}_{v_0}\cap\pi_{v_0}^{-1}(A)$), let \[d(n)=\begin{cases} d_{H(n_l)}(\mathcal{K}_{v_0}\cap\pi_{v_0}^{-1}(U)) & \textup{ when }n=n_l\text{ for }l\in\mathbb{Z}_+\\ d_{H(n)}(\mathcal{K}_{v_0}) & \textup{ otherwise}.\end{cases}\] For integers $n$ of the form $n=n_l$, let \begin{equation}\label{eqn:En}E_n=\{E_{n,1},\dots,E_{n,c(n)}\}\end{equation} be an arbitrary order $c(n)$ subset of $U$. For the set $\mathcal{N}$ of positive integers \textbf{not} of the form $n_l$, let $\{E_n\}_{n\in\mathcal{N}}$ be a sequence of order $c(n)$ subsets of $A(\mathbb{C}_{v_0})$ as in (\ref{eqn:En}) such that (\ref{eqn:liftlim}) holds for some choice of lifts $\widetilde{E_{n,i}}\in\mathcal{K}_{v_0}$ of $E_{n,i}$, where the limit is taken as $n\in\mathcal{N}$ tends to $\infty$. Since the sequence $\{E_n\}_{n\notin\mathcal{N}}$ does not equidistribute to $\mu_{\textup{can}}$, it follows from the above discussion that the sequence $\{d(n)\}_{n=1}^\infty$ does not converge. By the construction of the $d(n)$, we thus obtain (\ref{eqn:inequality}), as claimed. 

In particular, (\ref{eqn:corcapHn}) combined with (\ref{eqn:inequality}) yields \[\lim_{l\to\infty} d_{H(n_l)}(\mathcal{K}_{v_0}\cap\pi_{v_0}^{-1}(U))<1.\] Therefore there are constants $C_2=C_2(U)>0$ and $N_0=N_0(U)$ such that if $P_1,\dots,P_{c(n_l)}\in U$ and $n_l\ge N_0$, then \begin{equation}\label{C2} g_{n_l,v_0}\left(P_1,\dots,P_{c(n_l)}\right)\ge C_2.\end{equation} Suppose $P_1,\dots,P_{c(n_l)}\in A(L)$ where $[L:K]=D$, and let $n_w$ be local factors for the places $w$ above $v\in M_K$ characterized by the property that $\prod_{w\in M_K}|\alpha|_w^{n_w}=1$ for all $\alpha\in L^*$. Suppose that for $w\mid v_0$ and a fixed embedding of $A\times_KK_w$ into $A_{v_0}$, we have $P_i\in U$ for all $1\le i\le c(n_l)$. If $n_l\ge N_0$, then by (\ref{C1}) and (\ref{C2}), this implies that \begin{equation}\label{eqn:greentotallb} \sum_{w\in M_L}n_wg_{n_l,w}\left(P_1,\dots,P_{c(n_l)}\right)\ge \frac{C_2}{D}-\frac{C_1\log n_l}{n_l}.\end{equation} 

Let $P\in A(\overline{K})\setminus A(\overline{K})_{\textup{tors}}$ be defined over $K'$, where $[K':K]=D$ and let $\Phi_K$ be the component group of the N\'{e}ron model $\mathcal{A}_{v_0}$ of $A\times_KK_{v_0}$, with $\Delta:=|\Phi_K|$. Let $\{n_w\}_{w\in M_{K'}}$ be a collection of local factors for $K'$ such that $\prod_{w\in M_{K'}}|\alpha|_w^{n_w}=1$ for all $\alpha\in (K')^*$. Suppose without loss that $A$ has semistable reduction over $K$. Fix any $w\in M_{K'}$ lying above $v_0$. Let $\Phi_{K'}$ be the component group of the N\'{e}ron model of $A\times_{K}K_w$, and let $e_w$ be the ramification degree of $K_w/K_{v_0}$. By \cite[Theorem 5.3.3.9(2) and Remark 2.2.3.1]{HalleNicaise}, the exponents of $\Phi_{K'}$ and $\Phi_{K}$ satisfy \[\mathrm{exp}(\Phi_{K'})\mid e_w\cdot\textup{exp}(\Phi_K).\] (This is where we have used the assumption that $K_{v_0}$ has perfect residue field; Halle and Nicaise assume in \cite[Theorem 5.3.3.9]{HalleNicaise} that the residue field $k$ is algebraically closed. We apply their theorem to the maximal unramified extension of $K_{v_0}$, taking into account \cite[Remark 2.2.3.1]{HalleNicaise} as well as our assumption that $K_{v_0}$ has perfect residue field.) The exponent $\textup{exp}(\Phi_K)$ of $\Phi_K$ in turn satisfies \[\mathrm{exp}(\Phi_K)\mid\Delta\] for any place $w\in M_{K'}$ lying above $v_0$, where $e_w$ is the ramification index of $w$ above $v_0$. It follows that for any $w\in M_{K'}$ lying above $v_0$, the points $ke_w\Delta P$, $k\in\mathbb{Z}_+$ all retract to the point $\xi\in\Sigma_{w}=\Sigma_{v_0}$ corresponding to the identity component of the N\'{e}ron model. In particular, for any $n\in\mathbb{Z}_+$ and any $w\mid v_0$, \begin{equation}\label{eqn:keyretraction}W_n:=\{e_w\Delta P,\dots,e_w\Delta(2n^g+c(n))P\}\subseteq U.\end{equation} As we have assumed that $A$ is geometrically simple, we may apply Lemma \ref{lem:suitablemultiples} to the polarized dynamical system $(A,f)$ and the point $e_w\Delta P$. Lemma \ref{lem:suitablemultiples} implies that for any $n\ge1$, there is a $c(n)$-element subset \begin{equation}\label{eqn:Vn}V_n:=\left\{Q_1,\dots,Q_{c(n)}\right\}\end{equation} of $W_n$ such that for any lifts $\widetilde{Q}_i$ of $Q_i$, \[\textup{det}\left(\eta_j\left(\widetilde{Q_1},\dots,\widetilde{Q_{c(n)}}\right)\right)_{\eta_j\in H(n)}\ne 0.\] Thus by the product formula and (\ref{eqn:htrelation}), we have \begin{equation}\label{eqn:greentoht} \sum_{w\in M_{K'}}n_wg_{n,w}\left(Q_1,\dots,Q_{c(n)}\right)=\frac{1}{c(n)}\sum_{j=1}^{c(n)}\hat{h}_f(Q_j)=\frac{1}{c(n)}\sum_{j=1}^{c(n)}m\hat{h}_{\mathcal{L}}(Q_j)\end{equation} where $\hat{h}_f$ is the canonical height associated to $f$. 

For $n\in\mathbb{Z}_+$, write \[t_n=2n^g+c(n).\] Let $N_0$ be as in (\ref{C2}). Suppose that $\hat{h}_f(t_{N_0}e_w\Delta P)>\frac{C_2}{2D}$. Then \[\hat{h}_f(P)>\frac{C_2}{2t_{N_0}^2D\Delta^2 e_w^2}\ge\frac{C_3}{D^3}\] for some $C_3>0$, so there is nothing left to prove. Therefore we may assume that \[\hat{h}_f(t_{N_0}e_w\Delta P)\le \frac{C_2}{2D},\] i.e., that we may choose an integer $n_l\ge N_0$  in the indexing sequence defining (\ref{eqn:nj}) such that \begin{equation}\label{eqn:canhtbd} \hat{h}_f(ke_w\Delta P)\le \frac{C_2}{2D}\end{equation} for all $1\le k\le t_{n_l}$ and \begin{equation}\label{eqn:maxtnl} \hat{h}_f(t_{n_{l+1}}e_w\Delta P)>\frac{C_2}{2D}.\end{equation} For such a choice of $n_l$, and for $W_{n_l}$ and $V_{n_l}$ as in (\ref{eqn:keyretraction}) and (\ref{eqn:Vn}) respectively, the inequality (\ref{eqn:canhtbd}) yields $\hat{h}_f(Q_j)\le\frac{C_2}{2D}$ for all $1\le j\le c(n_l)$. Thus (\ref{eqn:greentoht}) implies \[\sum_{w\in M_{K'}}n_wg_{n_l,w}\left(Q_1,\dots,Q_{c(n_l)}\right)<\frac{C_2}{2D}.\] Combining this with (\ref{eqn:greentotallb}), which we may do because $Q_1,\dots,Q_{c(n_l)}\in U$ at $w$ by (\ref{eqn:keyretraction}) and we have assumed that $n_l\ge N_0$, we obtain \[\frac{C_2}{D}-\frac{C_1\log n_l}{n_l}<\frac{C_2}{2D}.\] We thus obtain \[n_l\le BD\log D\] for some constant $B$ independent of $D$. As \[\dim\,H^0(A,\iota^*(\mathcal{O}(n)))=O(n^{g})\] by \cite[Theorem 9.8(i)]{BoucksomEriksson}, this implies that \begin{equation*}\label{eqn:firstbound}t_{n_{l+1}}\le B'D^g\log^gD\end{equation*} for some constant $B'$ independent of $D$. In other words, by (\ref{eqn:maxtnl}), we have \[\hat{h}_f(e_wP)>\frac{C_2}{2D\Delta^2t_{n_{l+1}}^2}\ge\frac{C_3}{D^{2g+1}\log^{2g}\hspace{-0.5mm}D}\] for all $D\ge2$ and for some $C_3>0$ independent of $D$. From this we obtain \[\hat{h}_f(P)>\frac{C_3}{D^{2g+3}\log^{2g}\hspace{-0.5mm}D}\] whenever $D\ge2$. By the minimality of $m$ appearing in (\ref{eqn:htrelation}), applying (\ref{eqn:htrelation}) then gives \[\hat{h}_\mathcal{L}(P)>\frac{C}{D^{2g+3}\log^{2g}\hspace{-0.5mm}D}\] for a constant $C>0$ depending only on $A$ and $\mathcal{L}$.\end{proof}

\end{document}